\documentclass{amsart}

\usepackage{graphicx}
\usepackage{amssymb}
 \usepackage{amsthm}
 \usepackage{amsmath}
\usepackage{color}
\usepackage{amsrefs}
\usepackage[pagewise]{lineno}
 \usepackage{mathptmx}      
\newtheorem{theorem}{Theorem}
\newtheorem{remark}{Remark}
\newtheorem{proposition}{Proposition}
\newtheorem{corollary}{Corollary}
\newtheorem{lemma}{Lemma}
\newtheorem*{lemma:associativity}{Lemma \ref{le:boundvarphi}}

\newcommand{\R}{{\mathbb{R}}}
\newcommand{\N}{{\mathbb{N}}}
\newcommand{\cH}{{\mathcal{H}}}
\newcommand{\cV}{{\mathcal{V}}}

\title{ Qualitative Analysis of certain Reaction-Diffusion Systems of the FitzHugh-Nagumo type}

\begin{document}
\author{B. Ambrosio}
\address{Normandie Univ, UNIHAVRE, LMAH, FR-CNRS-3335, ISCN, 76600 Le Havre, France\\
The Hudson School of Mathematics, NYC}
\email{benjamin.ambrosio@univ-lehavre.fr}
\subjclass[2010]{...}

\begin{abstract}
This article aims to provide insights into the qualitative analysis of some
nonlinear Reaction-Diffusion (RD) systems arising in Neuroscience. We first
introduce a non-homogeneous FitzHugh-Nagumo (nhFHN) featuring excitability
and oscillatory properties. Then, we discuss the qualitative analysis of a toy model
related to nhFHN. In particular, we focus on the convergence of solutions of the toy
model toward different solutions (fixed point, periodic) and show the existence of a
cascade of Hopf bifurcations. Finally, we connect this analysis to the nhFHN system.
\end{abstract}


\maketitle
\section{Introduction}

This article is motivated by the qualitative analysis of the following RD system

\begin{equation}
 \left \{ 
     \begin{array}{rcl}
      \epsilon u_t&=& f(u)-v +d u_{xx}, \,(x,t)\in (a,b)\times (0,+\infty)\\
      v_t&=& u -c(x)\\
     \end{array}
      \right. 
      \label{FHNRD}
\end{equation}
with $f(u)=-u^3+3u$, $\epsilon$ small, and Neumann Boundary conditions (NBC). In the Neuroscience context, $u$ represents a potential and $v$ a recovery variable while $\Omega=(a,b)$ relates to the  spatial domain, typically the axon  of one single neuron. This system provides a simple model generating excitable and oscillatory behavior.  Thanks to the non-homogeneity in $c(x)$, it can feature both properties in one single PDE, see \cites{Amb-2009,Amb-2009-thesis,Amb-2017}. This is well suited for Neuroscience applications where it is relevant to characterize the response of an excitable tissue as the stimulus intensity or frequency varies. It also raises deep mathematical questions. Before delving into more details, let us recall some contextual facts which motivate the introduction of this non-homogeneous model. In 1952, see \cite{HH}, Hodgkin and  Huxley (HH) provided their seminal model, describing the electrical pulse propagation along the giant squid neuron axon. They were awarded with the 1963 prize in Physiology and Medicine (together with J. Eccles) for their work which included both a modeling part and physiological experiments. Of note, their original article contained a four ODE model as well as a spatially extended RD system (HH model) based on the cable analogy. We refer, for example to \cites{BookDay2001,BookErm2010,BookIzh2006} for further details on the model. Later, in 1961, see \cite{Fit-1961}, inspired by the HH model, R. FitzHugh introduced a two-dimensional ODE to provide a simpler model that can reproduce excitable and oscillatory behavior. He started from the Van der pol model (VDP), see \cite{VDP26}, known for its oscillatory behavior. On the other hand, in \cite{Bon-1948}, K.F. Bonhoeffer described the dynamics of a two-dimensional ODE system featuring oscillatory or excitable properties. Following the ideas of Bonhoeffer, Fitzhugh modified the VDP model to obtain a 2d equation able to exhibit the oscillatory and excitable dynamics observed in the HH model. He named it the BonHoeffer Van der Pol model. One year later, Nagumo \textit{et al.} provided an analog electrical circuit in \cite{Nagumo}. The model is now widely known as the FitzHugh-Nagumo model. With these ideas in mind, a simple model which can provide excitable and oscillatory properties is the diffusion-less underlying ODE system of \eqref{FHNRD} with a constant $c$. Indeed, it is well known that if $|c|<1$, the unique stationary point is unstable and that there exists a unique limit cycle that attracts all the non-constant trajectories. At $|c|=1$,  a supercritical Hopf bifurcation occurs, and for $|c|\geq 1$ the stationary point is globally asymptotically stable. Furthermore, in the latter case, the system is excitable because choosing an initial condition below a certain threshold will induce a large excursion in the phase plane, which is akin to a pulse for the first variable. Our motivation in studying the diffusive system \eqref{FHNRD}  is to investigate the qualitative behavior when $c(x)$ is set within a range of values giving raise to excitable and oscillatory dynamics. For example, will the signal propagate if some cells are oscillatory at the center of the domain and anywhere else the cells are excitable? This question, which relates to an interplay between excitability and oscillatory properties, has been partially  addressed in a series of papers, see \cites{Amb-2016,Amb-2009,Amb-2014}. This typically leads to a bifurcation path from a stationary state to the propagation of oscillations when peripheral cell's excitability increases. A complementary question is the characterization of the basin of attraction of patterns, see \cite{Amb-2016}. Studying the propagation of oscillations initiated at the center throughout the domain relates to the theory of traveling waves. While our approach here is quite different, it is worth mentioning some contributions of this theory that applies to the FHN equations.  Since the first studies on the Fisher-KPP equation \cites{Fis1937,Kol1937,McK1975}, the topic has indeed aroused considerable interest, see for example \cite{BookVolpert1994}. The particular case of the wave propagation phenomenon in the diffusive FHN, with $\Omega=\R$ has also been intensively studied for a few decades. One of the first contributions can be found in \cites{Rin1973,Rin1975}. In these articles, J. Rinzel  and coauthors, following the ideas in McKean \cite{McK1970}  studied a FHN RD system with piecewise linear nonlinearity (instead of the cubic nonlinearity). They provided explicit computations of periodic solutions, pulses, propagation speeds, and stability results. The case where the underlying ODE is bistable was also considered later in \cite{Rin1982a}. Around the same period, in a series of four articles, J.W. Evans provided a more theoretical analysis of a general model of nerve conduction, see  \cites{Eva1971,Eva1972a,Eva1972b,Eva1975}. In \cite{Eva1972a}, he defines the so-called Evans function, which would become an essential tool for traveling wave stability analysis see \cites{San2002,Bar2018}. In 1978, Rauch and Smoller, see \cite{Rau1978}, provided an analysis of FHN with $\Omega=\R$ which includes the well-posedness of the problem in various Banach spaces, as well as the stability of the stationary solution $(0,0)$ with various techniques (linearization, contracting rectangles...). A few years later, see \cite{Jon1984}, C. Jones, relying on the papers of Evans, provided a detailed analysis focusing on the stability of the traveling waves of the FHN diffusive equation.   Since then, many studies have been devoted to the characterization of different properties of traveling waves. This includes the proof of the existence of pulses of monotonic and periodic tales and the construction of solutions in a slow-fast context thanks to singular perturbation theory or via asymptotic expansions. Relevant articles to these matters are for example: \cites{Car1977,Has1976,Bez2017,Car2015,Car2016,Car2018,Cor2018,Erm1996,Jon1991,Kal1993,Guc2009,Guc2010,Den1991,Cha2007,Kru1997}. Those ideas and techniques are relevant to the present work, but our approach here is somewhat different. First, we work on a bounded interval subjected to NBC and not in $\R$. Moreover our first motivation is to understand the dynamics that can emerge from the inhomogeneity in $c$. This relates by many aspects to the simpler case treated in \cite{Amb-2014} where two coupled ODE FHN systems with different values of $c$ were studied. There is in fact several approaches to track the complex oscillations arising in nhFHN. The major contribution of the present article is to consider a toy model from which interesting oscillatory can be proved rigorously by projection in the eigenfunctions and to emphasize the analogy with nhFHN. Before to proceed further, it is also necessary to mention classical articles which consider RD systems related to \eqref{FHNRD} in a bounded domain and NBC. Those include for example \cites{Marion,BookTemam,BookRobinson} and references therein cited. One must remark however, that in comparison to the systems considered there, system \eqref{FHNRD} has two characteristics that, combined, bring technical difficulties:
\begin{enumerate}
\item the system is partially dissipative in the sense that only the first equation diffuses
\item there is no dissipative term of the form $-\delta v$ in the second equation 
\end{enumerate}

As mentioned before, we are particularly interested in bifurcation phenomena from which periodic solutions arise as well as in the stability of these solutions. To this end, we will first introduce a toy model which allows proving interesting results with explicit computations. Then, we will emphasize how this relates to system \eqref{FHNRD}. Technically, for \eqref{FHNRD}, due to the nonhomogeneity arising with $c(x)$, the spectrum analysis, which relies on the eigenvalues of $u_{xx}$ for the toy model, is replaced by a Sturm-Liouville analysis. The remaining of the article is divided as follows: we analyze the toy model in section two; section three is devoted to the analysis of \eqref{FHNRD}. The concluding remarks follow in section four.

\textbf{Notations and General Framework}\\

A classical approach is to work in the $L^2$ setting, see \cites{Amb-2009-thesis,BookRobinson,BookTemam,BookJLL,BookRothe}.
Accordingly, we set the following notations:

 \[\cH=L^2(a,b)\times L^2(a,b)\]
\[\cV=H^1(a,b)\times H^1(a,b) \mbox{ where } H^1(a,b) \mbox{ denotes the usual Sobolev space.}\]
\[||\cdot|| \mbox{ denotes the usual norm on } \cH.\]
 Subscripts $t$ and $x$ denote time and space derivatives, respectively. 
We will not write $(a,b)$ for simplicity if there is no ambiguity. 
For the reader's convenience, we recall a result for the wellposedness of \eqref{FHNRD}. We omit the proof which relies on the Galerkin method; see, for example, \cites{Amb-2009-thesis,BookRobinson,BookTemam,BookJLL}.
\begin{theorem}
 For initial conditions (IC)  $(u,v)(x,0)$ given in $\cH$, there exists a unique solution $(u,v)(x,t)$
 satisfying

\[(u, v)\in C(\R^+,\cH)\]
\[\forall \mu>0,\, \forall T>0,\,\forall t \in (\mu,t) u(t) \in L^\infty(a,b) \]
\[u\in L^2((0,T),H^1)\cap L^4((a,b)\times (0,T))\]
For every fixed $t$, the mapping between IC and the solution at time $t$ is continuous on $\cH$.\\
\end{theorem}
There results are valid for all the systems considered in this paper. 
\section{Analysis of a toy model}
In this section, we focus on the qualitative analysis of the following system:
\begin{equation}
\label{eq:NluvhopfRD}
 \left \{ 
     \begin{array}{rcl}
      u_t&=& \alpha u-u^3-v+u_{xx} \\
      v_t&=& u 
     \end{array}
      \right.     
\end{equation}

on the domain $(0,1)$ with NBC. System \eqref{eq:NluvhopfRD} provides a toy example allowing a cascade of Hopf-bifurcations. It is a simple RD system for which interesting analytical results can be obtained. Note that for $\epsilon=d=1$, and $c$ constant, system \eqref{eq:NluvhopfRD} results from \eqref{FHNRD} after a change of variables and taking out the square term in $u$.   Here, we focus on asymptotic behavior. We are particularly interested in the co-existence of solutions, toward which convergence
depends on initial conditions (IC). To this end, we will first give a detailed analysis of the linearized system of \eqref{eq:NluvhopfRD} around $(0,0)$. While the analysis might look as straightforward computations, it emphasizes the qualitative dynamics within the ODEs associated with the projection on eigensubspaces and details valid proofs of stability
even though zero belongs to the closure of the set of eigenvalues. After that, we provide an analysis of \eqref{eq:NluvhopfRD}. We prove that a global stability result persists for  $\alpha<0$.  For $\alpha>0$, small, while the $(0,0)$ becomes unstable, we explicit linear subspaces of $\cH$, for which the solution still evolves towards $0$. We also provide a more general local stability result.  Finally, we illustrate some of the dymanics with numerical simulations.\\

\subsection{The linear case}
Note first that $(0,0)$ is a constant solution of \eqref{eq:NluvhopfRD}. The linearized system around this point is given by: 

\begin{center}
\begin{equation}
\label{eq:luvhopfRD}
 \left \{ 
     \begin{array}{rcl}
      u_t&=& \alpha u-v+u_{xx} \\
      v_t&=& u 
     \end{array}
      \right. 
\end{equation}
\end{center}
on the domain $(0,1)$ with NBC. The spectral decomposition allows to provide a comprehensive analysis of \eqref{eq:luvhopfRD}. Classically, we set:
\[\varphi_0(x)=1, \mbox{ and }\forall k \in \N^* \, \varphi_k(x)=\sqrt{2}\cos(k \pi x).\] 
We recall that the family $(\varphi_k)_{k  \in \N}$ is an orthonormal basis of $L^2$, and that the functions $\varphi_k$ satisfy:
\[-(\varphi_k)_{xx}=\lambda_k \varphi_k\]
and
\[(\varphi_k)_{x}(0)=(\varphi_k)_{x}(1)=0,\]
with 
\[\lambda_k=k^2\pi^2.\]
Looking for solutions expressed as,
\[u(t)=\sum_{k=0}^{\infty}u_k(t)\varphi_k,\, v(t)=\sum_{k=0}^{\infty}v_k(t)\varphi_k\]
leads by projection on the eigenspace generated by $(\varphi_k,\varphi_k)$ to the resolution of the two-dimensional ODE $E_k$: 
 \begin{center}
\begin{equation}
\label{eq:luvhopfRD-E_k}
 (E_k)\left \{ 
     \begin{array}{rcl}
      u_{kt}&=& (\alpha_k-\lambda_k)u_k-v_k\\ 
      v_{kt}&=& u_k
     \end{array}
      \right. 
\end{equation}
\end{center}
The eigenvalues of the matrix
\[A_k=\begin{pmatrix}
\alpha-\lambda_k &-1\\
1&0
\end{pmatrix}\]
are 
\[\sigma_k^1=\frac{1}{2}\bigg(\alpha-\lambda_k-\sqrt{(\alpha-\lambda_k)^2-4}\bigg),\, \sigma_k^2=\frac{1}{2}\bigg(\alpha-\lambda_k+\sqrt{(\alpha-\lambda_k)^2-4}\bigg).\]
The following proposition summarizes the properties of $\sigma_k^1$ and $\sigma_k^2$.
\begin{proposition}
When $\alpha$  crosses $\lambda_k$ from left to right, $\sigma_k^1$ and  $\sigma_k^2$ cross the imaginary axis from left to right. Furthermore,
\[\lim_{k\rightarrow +\infty}\sigma_k^1=-\infty \mbox{ and } \lim_{k\rightarrow +\infty}\sigma_k^2=0^-, \]
with
\[\sigma_k^2=\frac{1}{\alpha-\lambda_k}+o(\frac{1}{\alpha-\lambda_k}).\]
\end{proposition}


In the two next theorems, we state the main results describing the behavior of \eqref{eq:luvhopfRD}.
\begin{theorem}
\label{th:luvhopfRDneg}

 For $\alpha<0$, for any initial condition $(u(\cdot,0),v(\cdot,0))$ in $\cH$, we have 
 \[\lim_{t \rightarrow +\infty}||(u,v)(t)||= 0.\]
\end{theorem}
\begin{proof}
For any fixed $N>0$, one can prove that there exists $\delta>0$ such that
\[\sum_{k=0}^N(|u_k(t)|^2+|v_k(t)|^2)\leq e^{-\delta t}\sum_{k=0}^N(|u_k(0)|^2+|v_k(0)|^2),\]
where $(u_k,v_k)(t)$ is the solution of $E_k$ with 
\[u_k(0)=\int_0^1u(x,0)\varphi_k(x)dx,\,v_k(0)=\int_0^1v(x,0)\varphi_k(x)dx .\]
Note that since $\sigma_k^2 \rightarrow 0$, our computations do not allow us to take $N= \infty$ and conserve the exponential decay. However, for any $\epsilon>0$ there exists $N$ large enough such that 
\[\sum_{k=N+1}^{+\infty}(|u_k(0)|^2+|v_k(0)|^2)<\frac{\epsilon}{2}.\]
Since for $N$ large enough and $k>N$ we have,
\[\frac{d}{dt}(|u_k(t)|^2+|v_k(t)|^2)\leq 2(\alpha-\lambda_k)|u_k(t)|^2\leq 0,\]
the following inequality holds:
\[\forall t>0, \, \sum_{k=N+1}^{+\infty}(|u_k(t)|^2+|v_k(t)|^2)<\frac{\epsilon}{2}.\]  
Combining the above results, we can deduce that for any $\epsilon>0$, there exists $T$ such that for $t>T$, 
\[||(u,v)(t)||<\epsilon.\]
\end{proof}
\begin{theorem}
\label{th:luvhopfRD}
Let $k\in \N^*$.\\
 For $\alpha=\lambda_k$, $(0,0)$ is a center for system $E_k$, a source for $E_l$ if $l<k$ and a sink for $E_l$ if $l>k$. Furthermore, if $u_l(0)=v_l(0)=0$ for $l\in \{0,...,k-1\}$ then 
 \[\lim_{t \rightarrow +\infty}||(u,v)(t)-\varphi_k(u_k(t),v_k(t))||= 0.\]
  Otherwise, 
  \[\lim_{t \rightarrow +\infty}| |(u,v)(t)||=+\infty. \]
For $ \lambda_k<\alpha<\lambda_{k+1}$, $(0,0)$ is a source for $E_l$ si $l\leq k$ and a sink for $E_l$ if $l>k$. Furthermore, if $u_l(0)=v_l(0)=0$ for $l\in \{1,...,k\}$ then 
\[\lim_{t \rightarrow +\infty}||(u,v)(t)||= 0.\]  Otherwise  
\[\lim_{t \rightarrow +\infty}||(u,v)(t)||=+\infty. \]
\end{theorem}
\begin{remark}
It is worth noting that for the linear approximation, the above simple computations allow to characterize eigensubspaces of IC leading to convergence toward the fixed point, periodic solutions, or infinity.
\end{remark}
The spectral decomposition of $L^2$ has allowed an exhaustive study of the asymptotic behavior. Next, we consider the nonlinear case and show that spectral decomposition remains useful for asymptotic qualitative analysis.

\subsection{The Nonlinear case}
We now consider the system \eqref{eq:NluvhopfRD}
\begin{equation*}
 \left \{ 
     \begin{array}{rcl}
      u_t&=& \alpha u-u^3-v+u_{xx} \\
      v_t&=& u 
     \end{array}
      \right. 
\end{equation*}

on the domain $(0,1)$ with NBC.\\
For $\alpha<0$, the fixed point $(0,0)$ is still attracting all the IC in $\cH$. Indeed, we have:
\begin{theorem}
\label{th:NluvhopfRD}
 For $\alpha<0$,  for all IC in $\cH$
 \[\lim_{t \rightarrow +\infty}||(u,v)(t)||= 0\]
\end{theorem}
\begin{proof}
The proof relies on LaSalle's principle, see \cites{BookTemam,BookCaz1998}. For the reader's convenience, we give some details of the rigorous proof. We consider arbitrary IC $(u(x,0),v(x,0)) \in \cH$. And consider the solution $(u^n=\sum_{k=0}^nu^n_k\varphi_k,\sum_{k=0}^nv^n_k\varphi_k))$ of the approximated problem:
\begin{equation*}
 \left \{ 
     \begin{array}{rcl}
      u^n_{kt}&=& \alpha u^n_k-\int_\Omega (\sum_{l=0}^nu^n_l\varphi_l)^3\varphi_kdx-v_k^n-\lambda_ku_k^n \,\,\, k\in \{0,...,n\}\\
      v^n_{kt}&=& u_k^n 
     \end{array}
      \right. 
\end{equation*}
with IC $u^n(x,0)=\sum_{k=0}^n\varphi_k\int_0^1u(x,0)\varphi_kdx,v^n(x,0)=\sum_{k=0}^n\varphi_k\int_0^1v(x,0)\varphi_kdx$.\\The solution of the problem is the limit as $n$ goes to infinity of $(u^n,v^n)$.
Next, first observe that, for fixed $n$,
\begin{equation*}
\begin{array}{rcl}
\frac{d}{dt}||(u^n,v^n)(t)||^2&=&\alpha|u^n|_{L^2}^2-\int_0^1(u^n)^4dx-\int_0^1((u^n)_x)^2dx\\
&\leq&0.
\end{array}
\end{equation*}
 Furthermore,
\begin{equation*}
\begin{array}{rcl}
\frac{d}{dt}||(u^n_x,v^n_x)(t)||^2&=&\alpha|u^n_x|_{L^2}^2-3\int_0^1(u^n)^2((u^n)_x)^2dx-\int_0^1(u^n_{xx})^2dx\\
&\leq&0.
\end{array}
\end{equation*}
which proves that the trajectories are bounded in $\cV$. Thanks to the compact injection from $H^1$ into $L^2$, we have compacity allowing us to apply LaSalle's principle. But here, the $\omega-limit$ set of the trajectory ensued from $(u^n(x,0),v^n(x,0)$ is reduced to $(0,0)$. It follows that 
\[\lim_{t \rightarrow +\infty}||(u^n,v^n)(t)||=0\]
 in $\cH$.\\
  Now, fix $\epsilon>0$. And fix $n$ such that 
  \[||(u^n(x,0)-u(x,0),v^n(x,0)-v(x,0))||<\frac{\epsilon}{2}\]
  A simple computation shows that for any $p>n$,
  \[\frac{d}{dt}||(u^n-u^p,v^n-v^p)||^2\leq 0\]
  which leads to, letting $p$ go to infinity to
    \[||(u^n(x,t)-u(x,t),v^n(x,t)-v(x,t))||<\frac{\epsilon}{2} \mbox{ for all } t\geq 0.\]
Which yields
\[\lim_{t \rightarrow +\infty}||(u,v)(t)||=0\] 
\end{proof}

trajectories distinct from $(0,0)$.
It follows that, for $\alpha>0$, $(0,0)$ becomes unstable for any solution constant in space different from $(0,0)$ will evolve towards the limit cycle of the $ODE$. However, this system allows for the construction of specific solutions of interest. 

\begin{theorem}
For $0<\alpha<\lambda_1$, if $u(x)=-u(1-x)$ and $v(x)=-v(1-x)$ then for all IC in $\cH$
 \[\lim_{t \rightarrow +\infty}||(u,v)(t)||= 0\]
\end{theorem}
\begin{proof}
By symmetry, $\int_0^1 u(x,t)dx=\int_0^1v(x,t)dx =0$. 
Then, we apply LaSalle's Principle as in the proof of \ref{th:NluvhopfRD}, within the positive invariant subspace defined by $\int_0^1 u(x,t)dx=\int_0^1v(x,t)dx =0$. Indeed, we have:
\begin{equation*}
\begin{array}{rcl}
\frac{d}{dt}||(u,v)(t)||^2&=&\alpha|u|_{L^2}^2-\int_0^1u^4dx-\int_0^1u_x^2dx\\
&=&\sum_{k=1}^\infty(\alpha-\lambda_k)u_k^2-\int_0^1u^4dx\\
&\leq& 0

\end{array}
\end{equation*}
and,
\begin{equation*}
\begin{array}{rcl}
\frac{d}{dt}||(u_x,v_x)(t)||^2&=&\alpha|u_x|_{L^2}^2-3\int_0^1u^2u_x^2dx-\int_0^1u_{xx}^2dx\\
&=&\sum_{k=1}^\infty(\alpha-\lambda_k)\lambda_ku_k^2-3\int_0^1u^2u_x^2dx\\
&\leq&0.
\end{array}
\end{equation*}
\end{proof}
\begin{remark}
Note that although $(0,0)$ is unstable for $0<\alpha<\lambda_1$, this theorem characterizes a set included in the basin of attraction of $(0,0)$. The arguments provided here apply to larger dimensions. For example, when the space dimension is 2, this set contains IC leading to pattern formation such as spirals (for $\alpha>\alpha_1$); see \cite{Amb-2016} where equation \eqref{FHNRD} was studied theoretically and numerically for $c=0$. 
\end{remark}

The following lemma proves that a large number of nonlinear terms have a zero integral.
\begin{lemma}
Let $k,m,n,l \in \N, k,m,n,l>0$, then
\[\int_0^1\varphi_k(x)\varphi_m(x)\varphi_n(x)dx\neq 0\]
if and only if 
\[k+m=n  \mbox{ or }  k+n=m  \mbox{ or }  m+n= k.\]
And 
\[\int_0^1\varphi_k(x)\varphi_l(x)\varphi_m(x)\varphi_n(x)dx\neq 0\]
if and only if one of the subscripts is the sum of the three others or the sum of two of them equals the sum of the other two.
\end{lemma}
The next theorem provides a local stability result for $0<\alpha<\lambda_1.$
\begin{theorem}
\label{th:NLHopf-stab}
For $0<\alpha<\lambda_1$,  there exists a sequence $(\mu_k)_{k\in \N}$ such that if  
\[(u_k(0),v_k(0))\in B(0,\mu_k)\]
then
\[\lim_{t\rightarrow +\infty}||(u(t)-u_0(t),v(t)-v_0(t))||=0,\]
 where $B(0,\mu_k)\subset \R^2$ is the ball of center $(0,0)$ and radius $\mu_k$. 
 
\end{theorem}
\begin{proof}

We can write:
\begin{equation}
\label{eq:NluvhopfRD-2-E_0}
 (E_0)\left \{ 
     \begin{array}{rcl}
      u_{0t}&=& \alpha u_0-u_0^3-v_0-3u_0\sum_{i=1}^{+\infty}u_i^2-g_0\\
      v_{0t}&=& u_k   
     \end{array}
      \right. 
\end{equation}

where 
\[|g_0|\leq 3\frac{\sqrt{2}}{2}\sum_{i=1}^{+\infty}|u_i|\sum_{j=1}^{+\infty}|u_j|^2\]
and for $k\geq 1$,
\begin{equation}
\label{eq:NluvhopfRD-2-E_kgeq1}
 (E_k)\left \{ 
     \begin{array}{rcl}
      u_{kt}&=& (\alpha-\lambda_k-3u_0^2)u_k-v_k-9\frac{\sqrt{2}}{2}u_0(\sum_{i=1}^{+\infty}u_iu_{k+i} )-g_k\\
      v_{kt}&=& u_k   
     \end{array}
      \right. 
\end{equation}

where 
\[|g_k|\leq \frac{7}{2}\sum_{i=1}^{+\infty}|u_i|\sum_{j=1}^{+\infty}|u_j|^2\]
It follows that the nonlinear terms are bounded by:
\[C\sum_{i=1}^{\infty}|u_i|\sum_{i=1}^\infty u_i^2,\]
where $C$ is a constant.
Thanks to the dynamics of each 2d ODE system $E_k$, one can ensure that for any arbitrary small $\epsilon$,  for all $k \in \N$,
there exists $\mu_k$ such that if $(u_k(0),v_k(0))$ belongs to $B(0,\mu_k)$  the following estimates are valid:
\[|u_k|\leq 2 \frac{\epsilon}{|\alpha-\lambda_k|}, \, |v_k|\leq \epsilon \]
Thanks to the values of $\lambda_k=k^2\pi^2$, this actually implies that the series $\sum |u_k|$ is convergent. Note that, we do not have anymore bounds in $H^1$, in this case, and we cannot apply the LaSalle's principle as previously. We need to construct arguments by hand. 
Note further that,
\[0.5\frac{d}{dt}\sum_{k=1}^{+\infty}(u_k^2+v_k^2)=\sum_{k=1}^{+\infty}\big((\alpha-\lambda_k-3u_0^2)u_k^2-9\frac{\sqrt{2}}{2}u_0u_k(\sum_{i=1}^{+\infty}u_iu_{k+i} )-g_ku_k\big)\]
which yields
\[0.5\frac{d}{dt}\sum_{k=1}^{+\infty}(u_k^2+v_k^2)\leq \sum_{k=1}^{+\infty}\big(\alpha-\lambda_k-3u_0^2)u_k^2\big)+9\frac{\sqrt{2}}{2}u_0\sum_{k=1}^{+\infty}\big( u_k^2\sum_{i=1}^{+\infty}|u_i| \big)+C\sum_{k=1}^\infty u_k^2\sum_{i=1}^{\infty}|u_i| \]
Thanks to the bound $ |u_k|\leq 2 \frac{\epsilon}{|\alpha-\lambda_k|}$ and the convergence of $\sum  \frac{1}{|\alpha-\lambda_k|}$ , 
  we obtain for a certain constant $C$:
  \[0.5\frac{d}{dt}\sum_{k=1}^{+\infty}(u_k^2+v_k^2)\leq \sum_{k=1}^{+\infty}\big(C\epsilon+\alpha-\lambda_k-3u_0^2)u_k^2\big) \]
 Therefore one can choose $\epsilon$ such that:
\[\frac{d}{dt}\sum_{k=1}^{+\infty}(u_k^2+v_k^2)\leq 0\]
as well as
\[\frac{d}{dt}\sum_{k=1}^{+\infty}(u_k^2+v_k^2)<0,\]
as long as,
\[\sum_{k=1}^{+\infty} u_k^2 >0.\]
This implies that $\sum_{k=1}^{+\infty}(u_k^2+v_k^2)$ converges to a real value as $t$ goes to infinity. Further arguments omitted here allow to prove that the limit is $0$.
 
\end{proof}

\subsection{Numerical simulations}
In this section, we provide a few numerical simulations for system \eqref{eq:NluvhopfRD}. 
Our simulations have been performed using our own $C^{++}$ program with an $RK4$ numerical scheme. We use
 a time step of $10^{-5}$ and a space step of $0.02$. We illustrate simulations for values of the parameter $\alpha \in \{1,15\}$. For each value of $\alpha$,  we show two pictures : one corresponding to a solution with symmetry which implies that $\int udx=\int v dx=0$ and another one corresponding to a solution for which those integrals are non zero.  
 
 In figure $1$, we simulate  system \eqref{eq:NluvhopfRD} for $0<\alpha=1<\lambda_1$. IC satisfy the symmetric condition $u(x)=u(1-x)$ and $v(x)=v(1-x)$. ore precisely, or this figure $u(x,0)=v(x,0)=0.5$ on $(0,0.5)$, and $u(x,0)=v(x,0)=-0.5$ on $(0.5,1)$. This implies a symmetry for the solution for all positive times.  According to theorem \ref{th:NluvhopfRD}, the solution converges toward $(0,0)$ in $\cH$. The observation of $u$ in panel $A$ illustrates this theoretical result. For $v$, as illustrated in panel B, the evolution is slower. In panel C, we represent two trajectories corresponding to the time evolution of $u$ for two fixed values of $x$. One at $x=0.04$ in red and the other at $x=0.96$ in blue.
 
In figure $2$, we simulate  again system \eqref{eq:NluvhopfRD} for $0<\alpha=1<\lambda_1$. But here, IC do not satisfy the symmetric condition $u(x)=u(1-x)$ and $v(x)=v(1-x)$. In this figure, we set $u(x,0)=v(x,0)=0.5$ on $(0,0.5)$, $u(x,0)=v(x,0)=0$ on $(0.5,1)$.   The simulations show that the solution $u$ reaches asymptotically a constant function in space with periodicity in time. There is also periodicity in time for $v$, but the evolution of the shape in space  is slower. Asymptotically, since solutions are constant in space, the qualitative behavior is given by the ODE diffusion-less system.

In figure $3$, we simulate  system \eqref{eq:NluvhopfRD} for $\lambda_1<\alpha=15<\lambda_2$.  IC satisfy $u(x)=u(1-x)$ and $v(x)=v(1-x)$ which implies that $\int udx=\int v dx=0 i.e.$ $u_0(t)=v_0(t)=0$.  We observe that the solution $u$ evolves non constantly in space with periodicity in time.

In figure $4$, IC do not satisfy the symmetric condition $u(x)=u(1-x)$ and $v(x)=v(1-x)$. We observe that the solution $u$ reaches a constant function in space with  periodicity in time. We can note the difference of the amplitude of the limit cycle with the previous simulation; for these IC, $u_0(t)$ and $v_0(t)$ are no longer zero. 

Lastly, in figure 5, we illustrate $u$ as a function of $x$ ant $t$.

\begin{figure}
\includegraphics[scale=0.4]{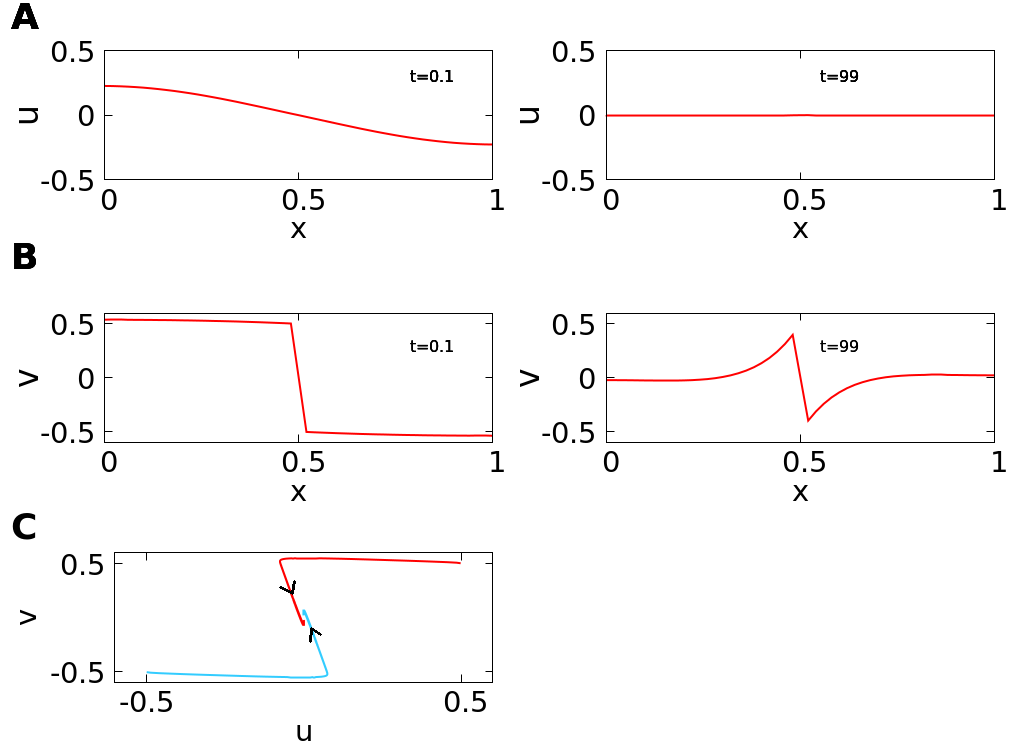}
\caption{Simulation of system \eqref{eq:NluvhopfRD} for $0=\alpha=1<\lambda_1$. IC satisfy the symmetric condition $u(x)=-u(1-x)$ and $v(x)=-v(1-x)$. We choose $u(x,0)=v(x,0)=0.5$ on $(0,0.5)$, $u(x,0)=v(x,0)=-0.5$ on $(0.5,1)$ .   According to theorem \ref{th:NluvhopfRD}, the solution converges toward $(0,0)$ in $\cH$. The panel A illustrates $u(x,t)$ for $x\in (0,1)$ and $t=0.1$ (left), $t=99$ (right). We observe that the solution $u$ reaches the constant function $0$ as predicted by the theory. The panel B illustrates $v(x,t)$ for $x\in (0,1)$ and $t=0.1$ (left), $t=100$ (right). Note that, since after some time, $u$ is close to $0$, according to the equation $v_t=u$, the evolution is slow. The last panel C shows the evolution in the time interval $(0,100)$ for fixed $x=0.02$ in red and $x=0.98$ in blue. Note the symmetry of the two trajectories }

\end{figure}

\begin{figure}
\includegraphics[scale=0.4]{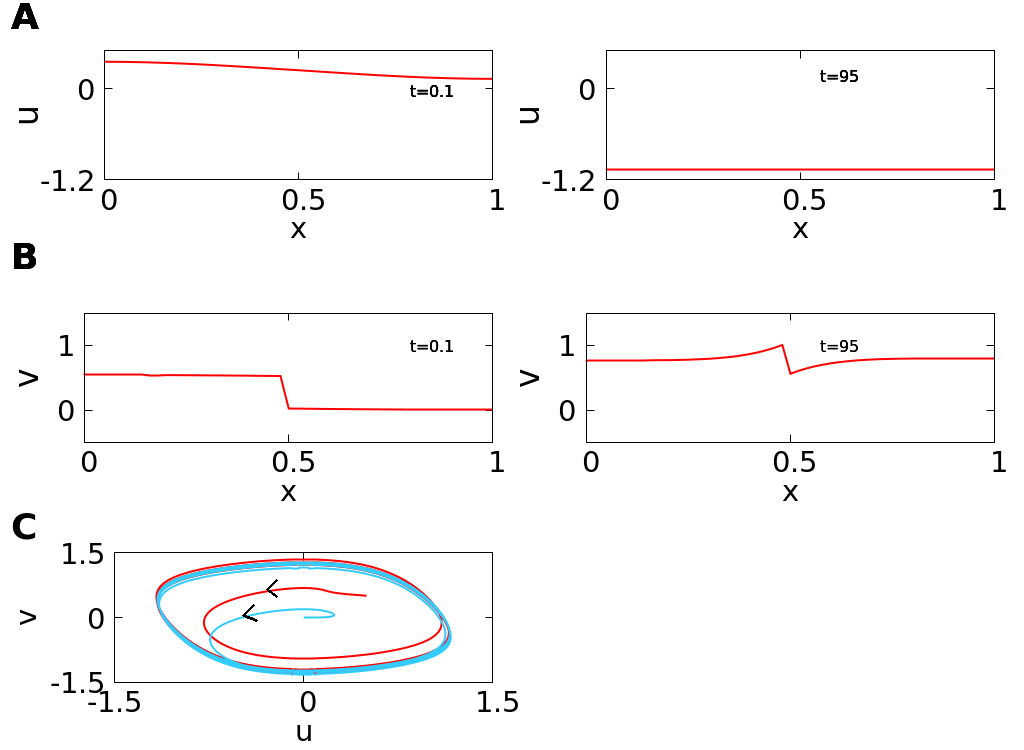}
\caption{Simulation of system \eqref{eq:NluvhopfRD} for $0=\alpha=1<\lambda_1$. IC do not satisfy the symmetric condition $u(x)=u(1-x)$ and $v(x)=v(1-x)$. Here, we set $u(x,0)=v(x,0)=0.5$ on $(0,0.5)$, $u(x,0)=v(x,0)=0$ on $(0.5,1)$ . The first row illustrates $u(x,t)$ for $x\in (0,1)$ and $t=0.1$ (left), $t=95$ (right). We observe that the solution $u$ reaches asymptotically a constant function in space. A look at the picture in the last row indicates a periodicity in time. The second row illustrates $v(x,t)$ for $x\in (0,1)$ and $t=0.1$ (left), $t=95$ (right). Complementary observation not illustrated here suggests that the shape of $v$ is moving slowly while its mean value is moving periodically fast. The last row shows the evolution in the time interval $(0,100)$ for fixed $x=0.02$ in red and $x=0.98$ in blue.}

\end{figure}
\begin{figure}
\includegraphics[scale=0.4]{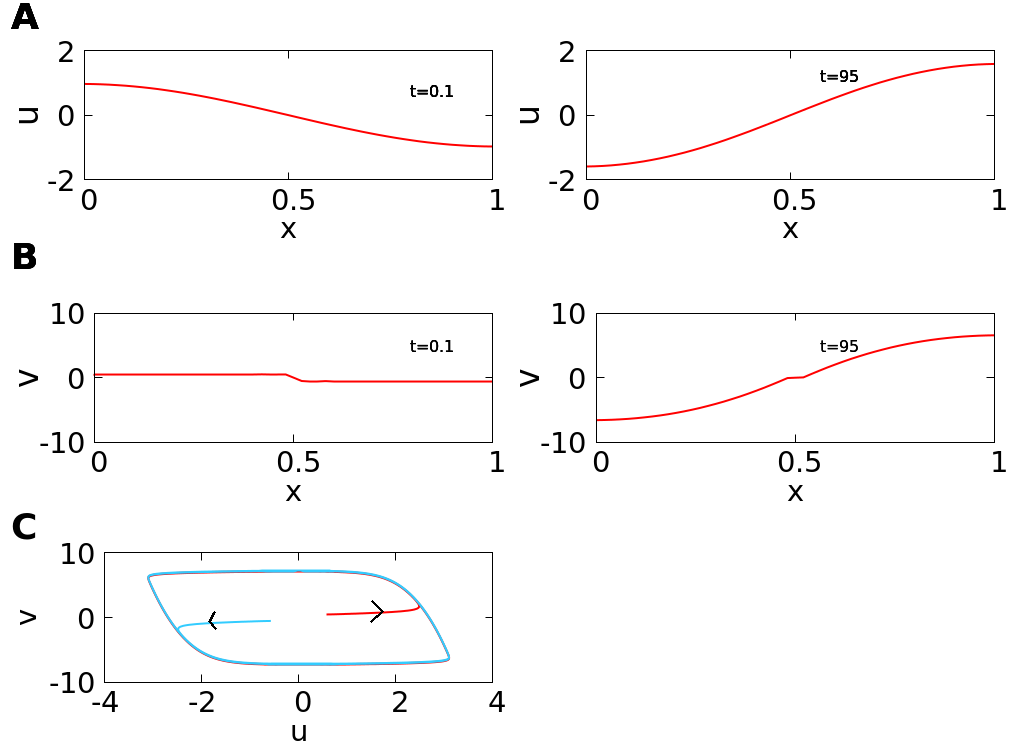}
\caption{The case: $\lambda_1<\alpha=15<\lambda_2$.  IC satisfy the symmetric condition $u(x)=u(1-x)$ and $v(x)=v(1-x)$. As in figure 1, we choose again,  $u(x,0)=v(x,0)=0.5$ on $(0,0.5)$, $u(x,0)=v(x,0)=-0.5$ on $(0.5,1)$. This implies that $u_0(t)=\int_0^1u(x,t)dx=v_0(t)=\int_0^1v(x,t)=0$.  Note however that for this value of $\alpha$, $(0,0)$ is now a source for system \eqref{eq:luvhopfRD-E_k} with $k=1$. The first row illustrates $u(x,t)$ for $x\in (0,1)$ and $t=0.1$ (left), $t=95$ (right). We observe that the solution $u$ evolves non constantly in space with periodicity in time. The second row illustrates $v(x,t)$ for $x\in (0,1)$ and $t=0.1$ (left), $t=100$ (right). The last row shows the evolution in the time interval $(0,100)$ for fixed $x=0.02$ in red and $x=0.98$ in blue. }

\end{figure}

\begin{figure}
\includegraphics[scale=0.4]{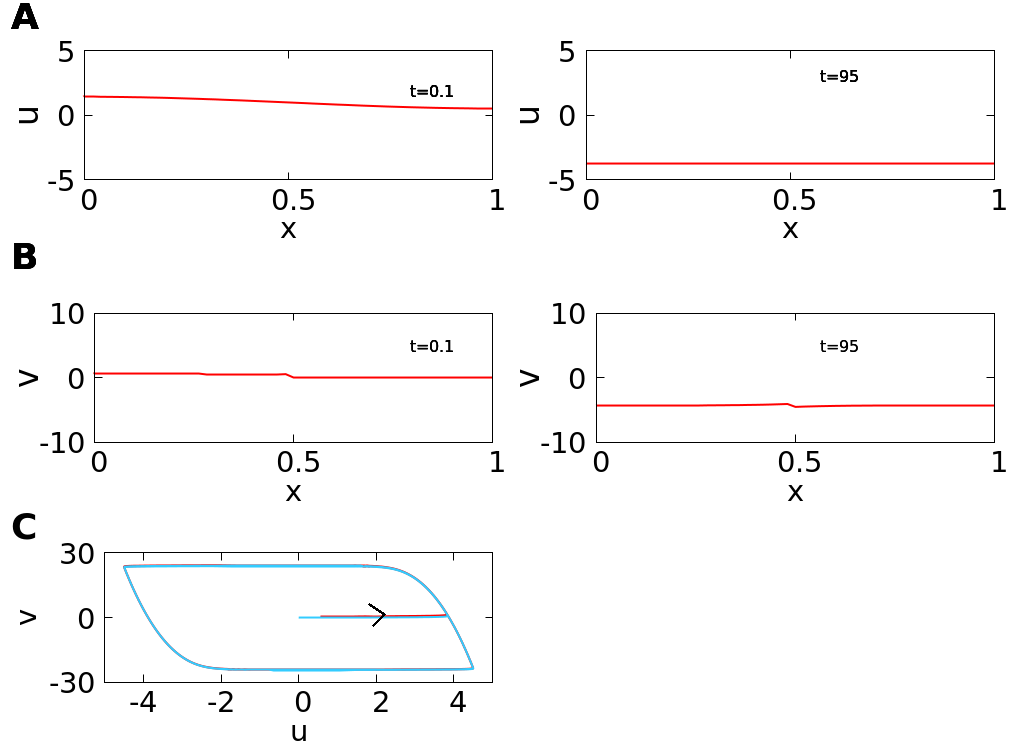}
\caption{The case: $\lambda_1<\alpha=15<\lambda_2$. IC do not satisfy the symmetric condition $u(x)=u(1-x)$ and $v(x)=v(1-x)$. We choose $u(x,0)=v(x,0)=0.5$ on $(0,0.5)$, $u(x,0)=v(x,0)=0$ on $(0.5,1)$ .   The first row illustrates $u(x,t)$ for $x\in (0,1)$ and $t=0.1$ (left), $t=95$ (right). We observe that the solution $u$ reaches a constant function in space. The picture in the last row indicates periodicity in time. The last row shows the evolution in the time interval $(0,100)$ for fixed $x=0.02$ in red and $x=0.98$ in blue. Note the difference with the amplitude of the limit cycle of figure 3; for these IC, $u_0(t)$ and $v_0(t)$ are no longer zero.  }
\end{figure}

\begin{figure}
\includegraphics[scale=0.4]{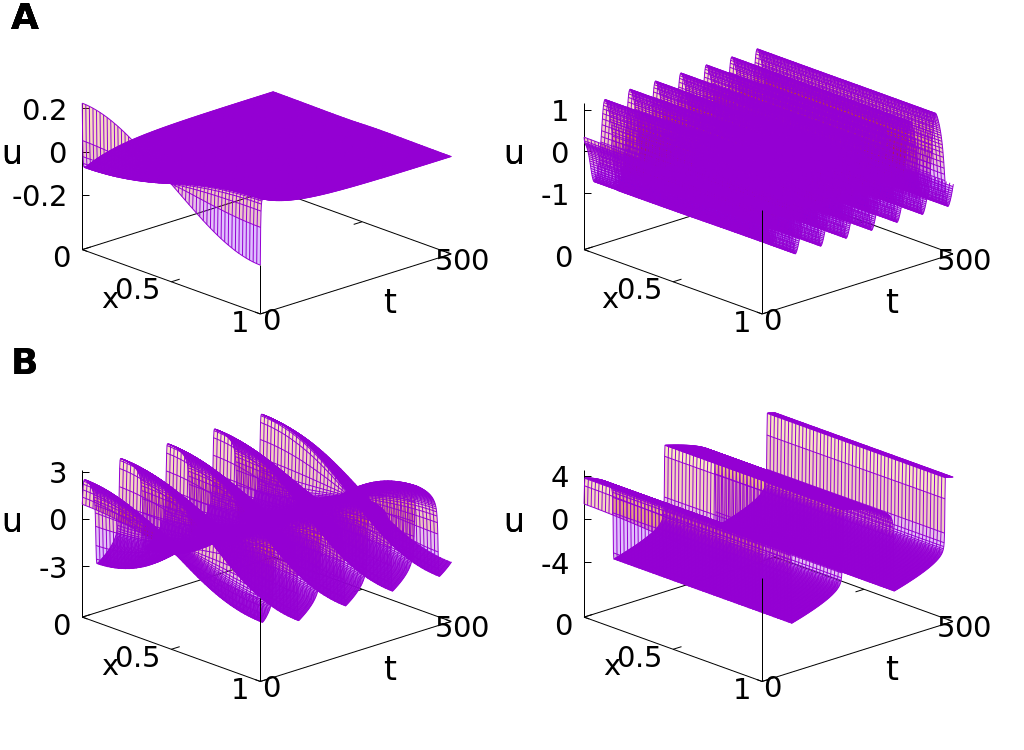}
\caption{This figure illustrates the behavior of $u(x,t)$. In each panel, we plot $u$ as a function of $x$ and $t$, for $(x,t)\in (0,1)\times(0,500)$. The solutions are those already illustrated in figures 1-4 with different plots. Illustrations in row A correspond to $0<\alpha=1<\lambda_1$, with zero integral (left) and non-zero integral (right) solutions. Illustrations in row B, correspond to $\lambda_1<\alpha=15<\lambda_2$ with zero integral (left) and non-zero integral (right) solutions. }
\end{figure}

\section{Analysis and discussion around the Nh-FHN model}

\subsection{Equation and Assumptions}
We now consider the system \eqref{FHNRD}:
\begin{equation*}
 \left \{ 
     \begin{array}{rcl}
      \epsilon u_t&=& f(u)-v +d u_{xx}, \,(x,t)\in (a,b)\times (0,+\infty)\\
      v_t&=& u -c(x)\\
     \end{array}
      \right. 
\end{equation*}
with $f(u)=-u^3+3u$, $\epsilon$ small, and Neumann Boundary conditions (NBC). We are interested in the case where the function $c$ take values lesser than $-1$ (excitable property for the diffusion-less system) on some parts of the domain $(a,b)$ and between $0$ and $1$ (oscillatory property for the diffusion-less system) on other parts of the domain. More specifically, we consider in this article a function $c$ which is oscillatory in a neighborhood of the center $\frac{a+b}{2}$ and excitable anywhere else; we further assume that excitability depends upon a given parameter $p>0$, see \cites{Amb-2009,Amb-2017}.   For sake of simplicity,  we assume that the function $c(x)$ (which depends on $p$) is regular and satisfies the following conditions:
\begin{eqnarray}
c(0.5(a+b))=0,& \,\label{eq:cl2}\\
\forall x \in (a,\frac{a+b}{2}),c'(x)>0 \, \forall x \in (\frac{a+b}{2},b), c'(x)<0 & \,\label{eq:cl3}\\
c'(a)=c'(b)=0,& \,\label{eq:cl4}\\
\forall x \in (a,b), x \neq 0, c(x) \mbox{ is a decreasing function of }p,& \, \label{eq:cl5}\\
\forall x \in (a,b), x \neq 0,\lim_{p\rightarrow 0}c(x)=0,& \,\label{eq:cl6}\\
\forall x \in (a,b), x \neq 0,\lim_{p\rightarrow +\infty}c(x)=-\infty.& \,\label{eq:cl7}
\end{eqnarray}

The first assumption \ref{eq:cl2} implies that the diffusion-less system at $x=0$ has a limit cycle. The second assumption \ref{eq:cl3} implies that $c$ admits a maximum at $x=0$, furthermore, it implies that the excitability strength decreases from the border toward the center.  The third assumption \ref{eq:cl4} is a compatibility condition with NBC. Finally, assumptions \ref{eq:cl5}-\ref{eq:cl7} relate to the strength of excitability with respect to $p$ for fixed $x$; the higher the value of $|c|$, the weaker the excitation's strength. 

The spectrum analysis of the linearized operator has been carried out in \cite{Amb-2017}.  We will  rely here on some of these results while following a different presentation. The stationary solution is given by:
\begin{equation}
\label{FHNRD1Dsta}
 \left \{ 
     \begin{array}{rcl}
       \bar{v}&=& f(\bar{u})+d \bar{u}_{xx},\\
      \bar{u}&=& c(x).\\
     \end{array}
      \right. 
\end{equation}
We rewrite  system \eqref{FHNRD1Dsta} around $(\bar{u},\bar{v})$. We obtain   

\begin{equation}
\label{FHNnh-2}
 \left \{ 
     \begin{array}{rcl}
      \epsilon u_t&=& f'(\bar{u})u+u_{xx}-v+0.5f''(\bar{u})u^2-u^3, \\
      v_t&=& u\\
     \end{array}
      \right. 
\end{equation}

We would like to proceed with projection on appropriate subspaces as in previous sections. To that end, we are interested in solutions of the following equation

\begin{equation}
\label{eq:SL}
f'(\bar{u})u+d u_{xx} =\lambda u
\end{equation}

Note that equation \eqref{eq:SL} is a regular Sturm-Liouville problem.
We have the classical following spectral theorem, see for example \cites{BookJost2013,BookTrudinger,BookTeschl}.
\begin{theorem}
\label{th:spec}
There exists an increasing sequence of real numbers $(\lambda_k)$ and an orthogonal basis $(\varphi)_{k\in \N}$ of $L^2(a,b)$ such that:
\begin{equation}
\begin{array}{rcl}
\label{EV}
d\varphi_{kxx}+f'(\bar{u})\varphi_k&=&-\lambda_k \varphi_k\\
\varphi_k'(a)=\varphi'_k(b)&=&0. 
\end{array}
\end{equation}
Furthermore, 
\[\lim_{k \rightarrow +\infty}\lambda_k=+\infty,\]
\begin{equation}
\label{Lam0}
\lambda_0 = \inf_{u \in H^2(a,b);\int_a^bu^2dx=1} d\int_a^bu_x^2dx-\int_a^b f'(\bar{u})u^2dx
\end{equation} 
\[\lambda_0\geq -3\]

\[\lambda_k=\frac{\pi^2k^2}{4a^2}+O(k)\]
(Weyl assymptotics)
\end{theorem}
The linearized system of \eqref{FHNnh-2} writes
\begin{equation}
\label{FHNnh-2-lin}
 \left \{ 
     \begin{array}{rcl}
      \epsilon u_t&=& f'(\bar{u})u+du_{xx}-v \\
      v_t&=& u.\\
     \end{array}
      \right. 
      \end{equation}
 and projection on the $kth$ subspace writes 
 \begin{equation}
\label{eq:Ek}
(E_k) \left \{ 
     \begin{array}{rcl}
      \epsilon u_{kt}&=& -\lambda_ku_k-v_k\\
      v_{kt}&=& u_k \\
     \end{array}
      \right. 
\end{equation}
The eigenvalues of the matrix associated to  $(E_k)$ are given by
\begin{equation}
\label{eq:eigenvalues}
 \left \{
 \begin{array}{rcl}
\sigma_1^k&=&\frac{1}{2\epsilon}\bigg(-\lambda_k-\sqrt{\lambda_k^2-4\epsilon} \bigg) \\
\sigma_2^k&=&\frac{1}{2\epsilon}\bigg(-\lambda_k+\sqrt{(\lambda_k^2-4\epsilon} \bigg)
\end{array}
\right.
\end{equation}     
The following result holds:
\begin{corollary}
\label{Prop:solinst}
 The number of eigenvalues with positive real part is finite while the number of eigenvalues with negative real part is infinite. If we assume that
\begin{equation}
\int_\Omega f'(\bar{u})dx>0,
\end{equation}
then  $\sigma_1^0$ and $\sigma_2^0$ have a positive real part. And,
\[\forall k\in \N, \, \forall i \in \{1,2\} \Re(\sigma_i^k)\leq \Re(\sigma_1^0)\leq\frac{3}{\epsilon}\]
\end{corollary}

The more elaborated theorem holds, see
\begin{theorem}
For $p$ small enough, $(0,0)$ is a source for $E_0$. For $p$ large enough, $(0,0)$ is a sink . There is an Hopf bifurcation:  there exists a value $p^*$ for which as $p$ crosses $p*$ from right to left,  the real part of $\sigma_1^0$ and $\sigma_2^0$ increases from negative to positive. The other eigenvalues remaining with negative real parts.
\end{theorem}
\begin{remark}
This theorem relies on the analysis of the Sturm-Liouville problem \eqref{eq:SL}. The first Hopf bifurcation occurs when $\lambda_0$ becomes negative. Successive Hopf bifurcations occur when the the successive $\lambda_k$ become negative. The number of possible Hopf bifurcation is bounded. The limit-case is given by $c(x)=0$.
\end{remark}

\subsection{Discussion and Comparison with the Toy model}
For the toy model, we provided the following results:
\begin{enumerate}
    \item  existence of a cascade of Hopf bifurcations
    \item existence of solutions converging toward the unstable stationary point $(0, 0)$ for $0<\alpha<\lambda_1$
    \item a local stability result for $0<\alpha<\lambda_1$.
\end{enumerate}
 The point 1 resulted from the spectral analysis of the linearized operator. The
points 2 and 3 were possible thanks to the property that  $\int \varphi_0\varphi_k=0$. Thanks to that property, despite the nonlinearity, the first projection spanned by $\varphi_0$ can be controlled in the others
subspaces. The symmetry of the reaction term was also
 necessary for the point 2. Dealing with Nh-FHN, one lose the points 2 and 3.
The point 1, i.e. the spectral analysis of the linearized operator remains partially
valid and gives good insights. Since an infinite number of eigenvalues is negative, looking at the behavior through the
 projection into a finite space spanned by the eigenfunctions of the Laplacian provides an interesting way to
gain qualitative insights. Beyond that natural perspective, a question that arises is the existence of initial conditions leading to the convergence toward the stationary solution $(c,f(c))$ when the first Hopf bifurcation arises. Such a result was proven for the toy model. These questions are to be addressed in a forthcoming article.

\subsection{Numerical simulations}
For the numerical simulations, we set $a=-1$, $b=1$, $d=1$ and 
\begin{equation*}
c(x)=p(x^4-2x^2).
\end{equation*}
Note that $c(0)=0$ which corresponds to oscillatory behavior for the ODE system, while $c(a)=c(b)=-p<-1$ which corresponds to a stable steady state for the ODE. We simulate equation \eqref{FHNRD} on $(a,b)$ with an explicit scheme of Runge-Kutta 4 type, with a time step of $10^{-5}$. The value of  $\epsilon$ is set to $0.1$.
Note that the dynamics depend on the size of the domain and the diffusion coefficient $d$. Here, we illustrate numerical simulations with these parameters fixed. We vary the parameter $p$.  We plan to proceed to a further qualitative numerical study in a forthcoming article. Note also that some complementary results have been obtained for a chain with discrete kicks in \cite{Amb-2020-2}.  We illustrate here solutions for three values of the parameter $p$.  If $p=2.2$, the solution converges toward a stationary solution; see figures 6 and 9. As $p$ is decreased, we observe waves. The amplitude of the waves depend on the parameter $p$, see figures 7,8,9 ($p=2$ and $p=1.9$).

\begin{center}
\begin{figure}
\includegraphics[scale=0.4]{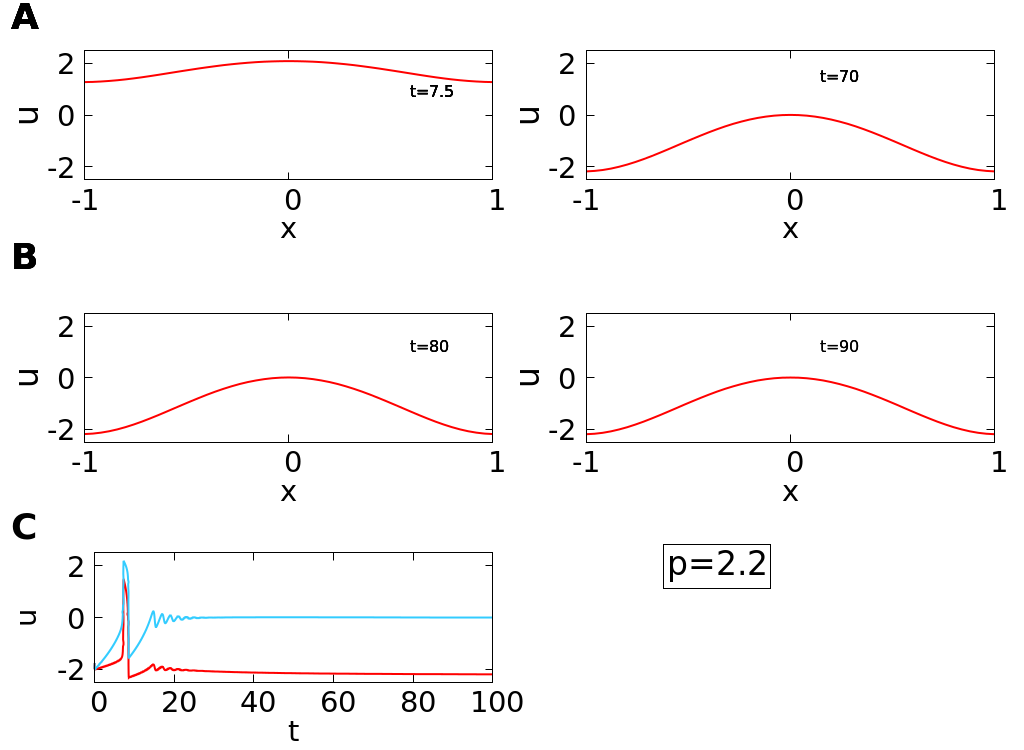}

\caption{Simulations of system \eqref{FHNRD} with $p=2.2$.  IC are set to a constant $(-1.5,2)$.  The first row illustrates $u(x,t)$ for $x\in (-1,1)$ and $t=7.5$ (left), $t=70$ (right).  The second row illustrates $u(x,t)$ for $x\in (-1,1)$ and $t=80$ (left), $t=90$ (right). The last row represents $u(x,t)$ for $t \in (0,100)$ and and two fixed values of $x$ ($x=-1$ in red,  and $x=0$ in blue). These numerical illustrations suggest that the solution evolves toward the stationary solution $(c,f(c)+c_{xx})$.  }

\end{figure}
\end{center}

\begin{center}
\begin{figure}
\includegraphics[scale=0.4]{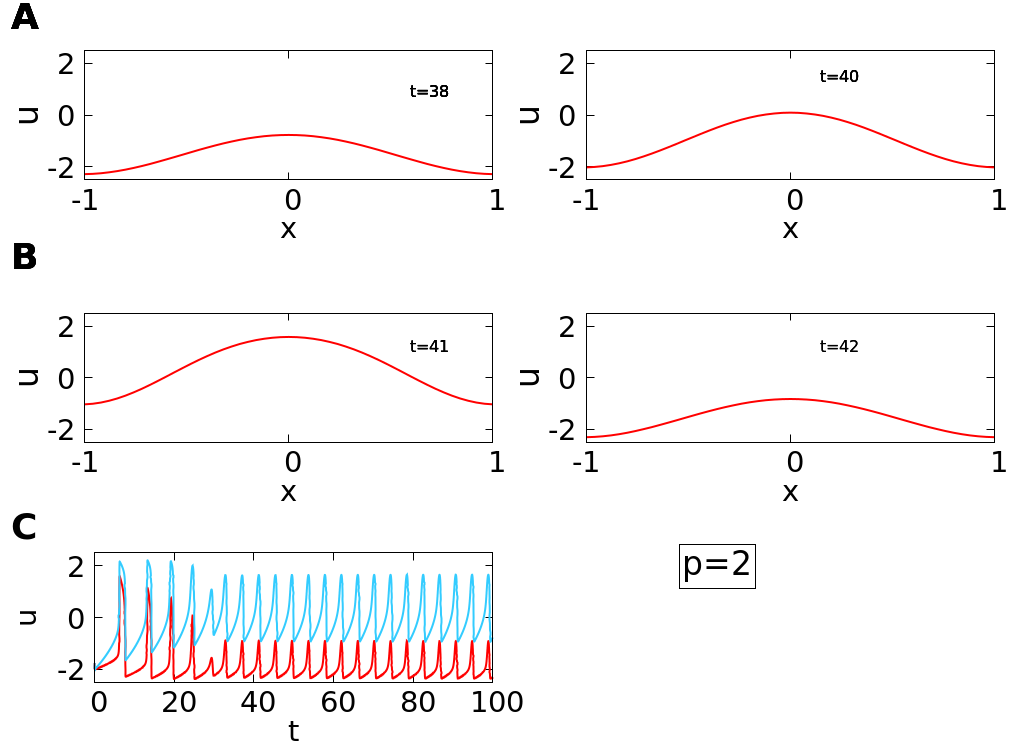}

\caption{Simulations of system \eqref{FHNRD} with $p=2.0$.  IC are set to a constant $(-1.5,2)$.  The first row illustrates $u(x,t)$ for $x\in (-1,1)$ and $t=38$ (left), $t=40$ (right).  The second row illustrates $u(x,t)$ for $x\in (-1,1)$ and $t=41$ (left), $t=42$ (right).  The last row represents $u(x,t)$ for $t \in (0,100)$ and and two fixed values of $x$ ($x=-1$ in red,  and $x=0$ in blue). These numerical illustrations suggest that the solution evolves toward a periodic solution in time. The first two rows A and B correspond approximately to snapshots along a time period.}

\end{figure}
\end{center}
\begin{center}
\begin{figure}
\includegraphics[scale=0.4]{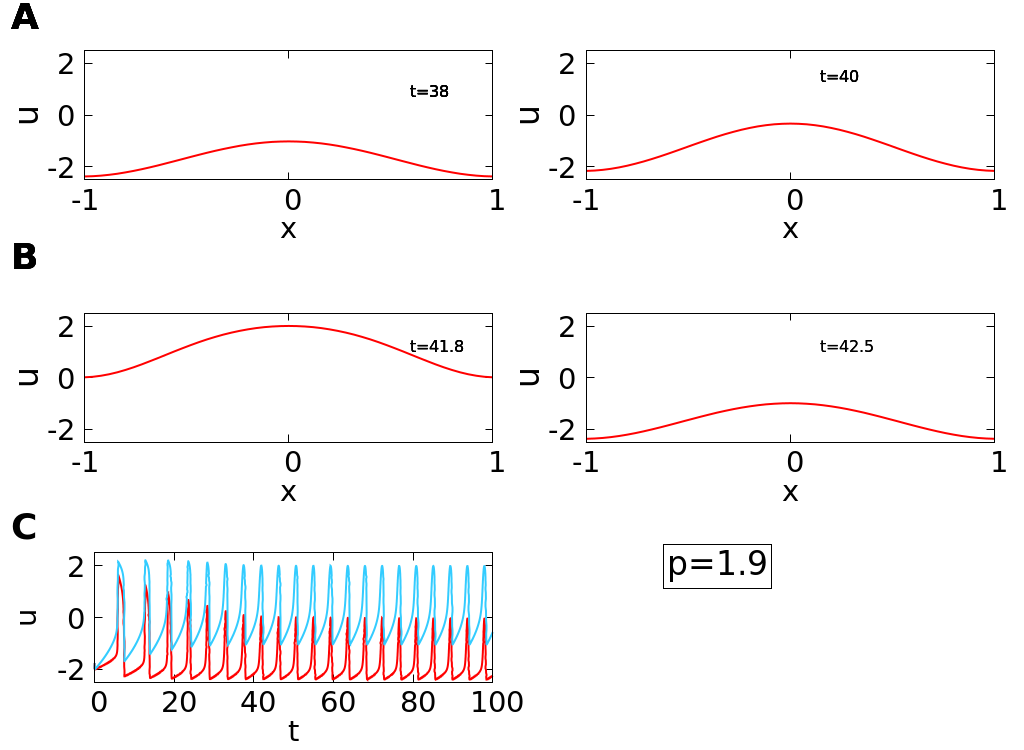}

\caption{Simulations of system \eqref{FHNRD} with $p=2.0$.  IC are set to a constant $(-1.5,2)$.  The first row illustrates $u(x,t)$ for $x\in (-1,1)$ and $t=38$ (left), $t=40$ (right).  The second row illustrates $u(x,t)$ for $x\in (-1,1)$ and $t=41.8$ (left), $t=42.5$ (right).  The last row represents $u(x,t)$ for $t \in (0,100)$ and and two fixed values of $x$ ($x=-1$ in red,  and $x=0$ in blue). These numerical illustrations suggest that the solution evolves toward a periodic solution in time. The first two rows A and B correspond approximately to snapshots along a time period. Note the difference of amplitude and frequency with solutions represented in figure 7. }

\end{figure}
\end{center}

\begin{center}
\begin{figure}
\includegraphics[scale=0.4]{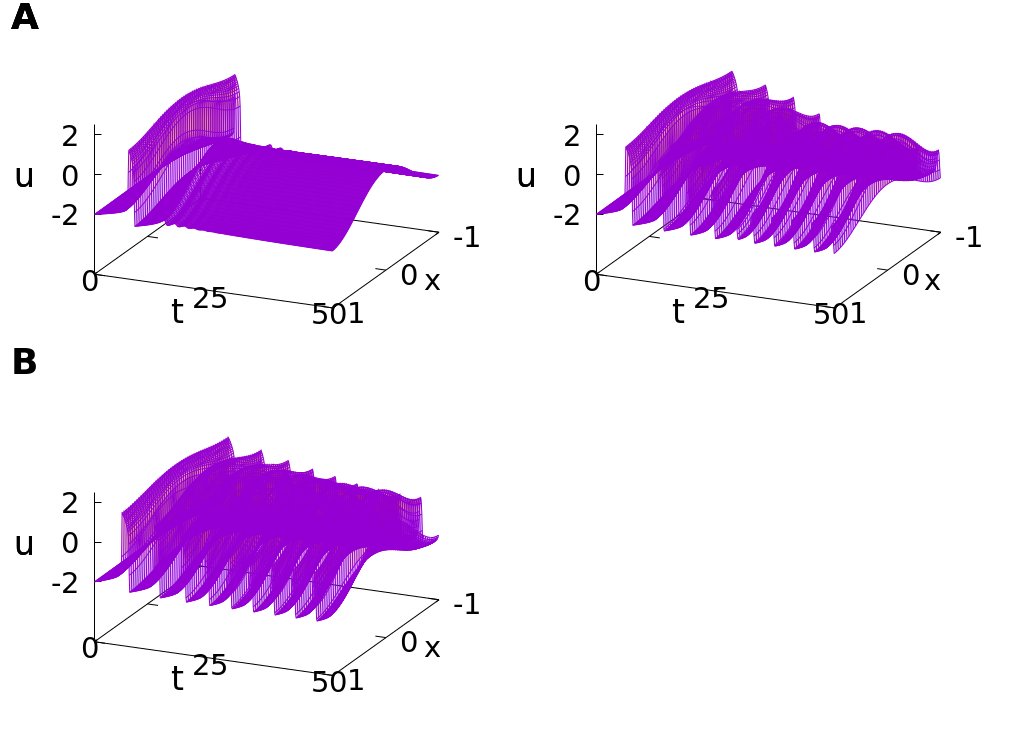}

\caption{Simulations of system \eqref{FHNRD} with $p=2.2$ (panel A left side), $p=2$ (panel A right side) and $p=1.9$ (panel B).  IC are set to the constant $(-1.5,2)$.  In each panel, we plot $u$ as a function of $x$ and $t$, for $(x,t)\in (-1,1)\times(0,50)$.  }

\end{figure}
\end{center}

\section{Concluding remarks and Perspectives}
In this article, we studied a few RD systems with a cubic nonlinearity arising from a Neuroscience context. We provided significant theoretical and numerical qualitative analysis of such infinite-dimensional systems highlighting dynamics arising under the stable regime and near  Hopf-Bifurcations. We also emphasized the dependence on initial conditions. As for the nature of the solutions observed, our work relates with previous studies on traveling waves. Our approach is however different by may aspects as this was discussed in introduction. As for the biological motivation, it goes back to the HH RD model and the oscillatory and excitability properties featured in the FitzHugh-Nagumo model. Our first contribution was to consider a non-homogeneous FHN model (Nh-FHN) as simple as possible. We introduced the inhomogeneity with a function $c$ depending upon the space variable $x$. This dependence allows to introduce both excitable and oscillatory properties within the single PDE. Note that in this context, the function $c(x)$  indicates the degree of excitability at the point $x$. Naturally, an immediate application is to provide a model for propagation of waves trough the axon of a single neuron, as this was the case in the original HH model. More generally, Nh-FHN is a model for electrical wave propagation along excitable tissue. It is worth noting that RD excitable systems derivated from FHN type systems have played an important role to model the electrical activity in heart. Spirals, scroll-waves and patterns have been used to characterize arrhythmias; see for example \cites{Maj2011,Pra2015,Prav2021} and references therein cited. There is also a large literature about numerical models used to develop techniques of low voltage fields to induce defibrillation. Those are based on RD ionic current models, see for example \cites{Are2007,Tra2014}.  The study of waves, spirals and patterns of electrical cortical activity has also been a very active field of research this recent years, see for example \cites{Mul2018,Mey2022}. In this context, a better knowledge of the underlying mathematical mechanisms
from which waves arise appears to be relevant. Here, we illustrated for example how waves may or may not exist depending on the initial conditions. Other simulations not illustrated here show also that the size of the domain an the diffusivity has an impact on the qualitative dynamics of the observed waves. 
In forthcoming articles, we plan to investigate further these phenomena. This includes the comparison between the solutions observed for the full PDE and finite dimensional solutions, a systematic description of the qualitative dynamics of the observed waves and how the results provided here extend to higher space dimensions.

\section*{Supplementary material}
The codes used for simulations are available at https://github.com/benjamin-ambrosio/EvEqCoTh2023ToyModel and https://github.com/benjamin-ambrosio/EvEqCoTh2023NhFHN 
\section*{Acknowledgments}
I would like to thank Region Normandie France, the 
ERDF (European Regional Development Fund)  project 
XTERM, IEA CNRS project  IEA00134 for funding.   

\bibliography{Ambrosioetal}

\end{document}